\documentclass[10pt]{amsart}
\usepackage{amstext}
\usepackage{amssymb,latexsym}

\def \diam {\text{diam}}

\newtheorem{thm}{Theorem}[section]
\newtheorem*{thmM}{Main Theorem}

\newtheorem{remark}[thm]{Remark}
\newtheorem{lemma}[thm]{Lemma}

\begin{document}
\title{The Radio Number of $C_{n} \square C_{n}$}
        \date{\today}
\author{Marc Morris-Rivera}
\address{Department of Mathematics, California State University Sacramento, Sacramento, CA} \email{mmorris@clunet.edu}
\author{Maggy Tomova}
\address{Department of Mathematics, University of Iowa, 14 MacLean Hall, Iowa City, IA 52242-1419} \email{mtomova@math.uiowa.edu}
\author{Cindy Wyels}
\address{Department of Mathematics, California State University Channel Islands, 1 University Dr., Camarillo, CA  93012} \email{cindy.wyels@csuci.edu}
\author{Aaron Yeager}
\address{Mathematics Department, University of Missouri, Columbia, MO 65211}
\email{aaronyeager@hotmail.com}


%This is supposed to change the running header when the authors list is too long.
\renewcommand{\shortauthors}{Morris-Rivera, Tomova, Wyels, Yeager}

\maketitle
\begin{abstract}
Radio labeling is a variation of Hale's channel assignment problem,
in which one seeks to assign positive integers to the vertices of a
graph $G$ subject to certain constraints involving the distances
between the vertices. Specifically, a radio labeling of a connected
graph $G$ is a function $c:V(G) \rightarrow  \mathbb Z_+$
%\{1, 2, \dots \}$
such that
$$d(u,v)+|c(u)-c(v)|\geq 1+\text{diam}(G)$$ for every two distinct
vertices $u$ and $v$ of $G$ (where $d(u,v)$ is the distance between $u$ and $v$).  The span of a
radio labeling is the maximum integer assigned to a vertex. The
radio number of a graph $G$ is the minimum span, taken over all
radio labelings of $G$. This paper establishes the radio number of
the Cartesian product of a cycle graph with itself (i.e. of
$C_n\square C_n$.)

\vspace{.1in}\noindent \textbf{2000 AMS Subject Classification:} 05C78 (05C15, 05C38)
\end{abstract}

\section{Introduction}
Radio labeling is derived from the assignment of radio frequencies
(channels) to a set of transmitters.  The frequencies assigned
depend on the geographical distance between the transmitters: the
closer two transmitters are, the greater the potential for
interference between their signals. Thus when the distance between
two transmitters is small, the difference in the frequencies
assigned must be relatively large, whereas two transmitters at a
large distance may be assigned frequencies with a small difference.

The use of graphs to model the ``channel assignment" problem was
first proposed by Hale in 1980 \cite{Hale}; Chartrand et al
introduced the variation known as radio labeling in 2001
\cite{CEHZ}.

In the graph model of the channel assignment problem, the vertices
correspond to the transmitters, and graph distance plays the role of
geographical distance. We assume all graphs are connected and
simple. The \emph{distance} between two vertices $u$ and $v$ of a
graph $G$, $d(u,v)$, is the length of a shortest path between $u$
and $v$. The \emph{diameter} of $G$, $\diam(G)$, is the maximum
distance, taken over all pairs of vertices of $G$. A \emph{radio
labeling} of a graph $G$ is then defined to be a function $c:V(G)
\to \mathbb Z_+$ satisfying
$$d(u,v)+|c(u)-c(v)|\geq 1+\text{diam}(G)$$
for all distinct pairs of vertices $u,v \in V(G)$. The \emph{span}
of a radio labeling $c$ is the maximum integer assigned by $c$. The
\emph{radio number} of a graph $G$, $rn(G)$, is the minumum span,
taken over all radio labelings of $G$\footnote{We use the convention, established in \cite{CEHZ}, that the co-domain of a radio labeling is $\mathbb Z_+ = \{1, 2, \dots \}$. Some authors use $\{0, 1, 2, \dots \}$ as the co-domain; radio numbers specified using the non-negative integers as co-domain are one less than those determined using the positive integers.}.

%Some authors let $\textbf{N}$ include $0$, with the result that radio numbers using this definition are one less than radio numbers determined using the positive integers.}.

We focus on Cartesian products of cycles. We remind the reader that
the cycle graph of order $n$, $C_n$, may be represented with vertex
set $V(C_n) = \{v_1, v_2, \dots , v_n\}$ and edge set
$E(C_n)=\{v_1v_2, v_2v_3, \dots , v_{n-1}v_n,v_nv_1\}$.
The diameter of $C_n$ is $\left\lfloor \frac n2 \right\rfloor$.

The Cartesian product of two graphs $G$ and $H$ has vertex set $V(G
\square H) = V(G) \times V(H) = \{(g,h)\,|\, g \in V(G) \text{ and
}h\in V(H)\}$. The edges of $G\square H$ consist of those pairs of
vertices $\{(g,h),\,(g',h')\}$ satisfying $g=g'$ and $h$ is adjacent
to $h'$ in $H$ or $h=h'$ and $g$ is adjacent to $g'$ in $G$. We note
that $C_n \square C_n$ has $n^2$ vertices, and $\diam(C_n \square
C_n)=2\left\lfloor \frac n2 \right\rfloor$.

As Liu and Zhu write, ``It is surprising that determining the radio number seems a difficult problem even for some basic families of graphs." \cite{LiuZhu} In fact, as of this writing, the only families of graphs for which the radio number is known are paths and cycles \cite {LiuZhu} and the squares of paths and cycles \cite{LiuXieP, LiuXieC}; wheels and gears \cite{gears}, and some generalized prisms \cite{prisms}. Meanwhile, bounds for the radio numbers of trees \cite{Liu}, ladders \cite{ladders}, and square grids \cite{grids} have been identified, while the radio number of cubes of the cycles $C_n^3$ for $n\leq 20$ and $n \equiv 0$, $2$, or $4 \pmod 6$ is known \cite{SR}.

The main result of this paper establishes the radio number of the Cartesian product of the $n$-cycle with itself:
\begin{thmM}\label{thmM:main} Let $n$ be a nonnegative integer. Then
\[
rn(C_n \square C_n)=
\begin{cases}
\frac{n^2-2}2(k+2)+2, &\text{for $n=2k$,}\\
\frac{n^2-1}2(k+1)+1, &\text{for $n=2k+1$.}
\end{cases}
\]
\end{thmM}
This is the first fully determined radio number for a family of graphs that is itself a Cartesian product of graphs. As such, it provides evidence for use in considering an interesting question: how is the radio number of a graph product related to the radio numbers of the factors? It is also possible that the labeling algorithms used to establish an upper bound for the radio number may be adapted to serve the same purpose for other toroidal graphs.

We prove the main theorem in two steps. First we provide the lower bound for $rn(C_n \square C_n)$ in Section 2. In Section 3, we define a radio labeling of $C_n \square C_n$; the span of this labeling is equal to the lower bound, thus establishing the radio number of $C_n \square C_n$. Finally, we return to the question of the relationship of $rn(G\square H)$ to $rn(G)$ and $rn(H)$ in Section 4, by examining $rn(C_n \square C_n)$ and $rn(C_n)^2$, as well as $rn(K_m \square K_n)$ and $rn(K_m)\cdot rn(K_n)$.

\section{Lower Bound}
The lower bound for $rn(C_n\square C_n)$ is reached in three steps.
First we examine the maximum possible sum of the pairwise distances
between any three vertices of $C_n\square C_n$. We use this maximum sum to
establish a minimum possible ``gap" between the $i^{th}$ and
$(i+2)^{nd}$ largest labels. Using 1 for the smallest label and
taking the size of the gap into account then provides a lower bound
for the span of any labeling.

We provide the details of this approach for $C_{2k}\square C_{2k}$
in Lemmas \ref{distance even} and \ref{Gap lemma even} and Theorem
\ref{lowerbound2k}. As the logic of the proofs of the corresponding
results for $C_{2k+1}\square C_{2k+1}$ is identical, we leave the
details of Lemma \ref{distance odd}, Lemma \ref{Gap lemma odd}, and
Theorem \ref{lowerbound2k+1} to the reader.
%***********************distance lemma for C_{2k}******************************
\begin{lemma}\label{distance even} Let $u,v,w \in V(C_{2k} \square C_{2k})$. Then
$d(u,v)+d(v,w)+d(u,w)\leq 2\,\diam(C_{2k}\square C_{2k})$.
\end{lemma}
\begin{proof} Express $u$, $v$, and $w$ via their component vertices,
i.e.,as $u=(x_1,y_1)$, $v=(x_2,y_2)$, and $w=(x_3,y_3)$, where
$x_i$ and $y_i$, $i = 1, 2, 3$ are all vertices of $C_{2k}$.  Then
$$\begin{aligned}
d(u,&v)+d(v,w)+d(u,w)\\
&=d\big((x_1,y_1),(x_2,y_2)\big)+d\big((x_2,y_2),(x_3,y_3)\big)+d\big((x_1,y_1),(x_3,y_3)\big)\\
&=d(x_1,x_2)+d(x_2,x_3)+d(x_1,x_3)+d(y_1,y_2)+d(y_2,y_3)+d(y_1,y_3).
\end{aligned}
$$
In taking shortest paths between $x_1$, $x_2$, and $x_3$ (all in $C_{2k}$),
one never need take more steps than those necessary to completely traverse $C_{2k}$, i.e.
$$d(x_1,x_2)+d(x_2,x_3)+d(x_1,x_3)\leq 2k.
$$
The same is true of the sum of the pairwise distances between
vertices $y_1$, $y_2$, and $y_3$. Thus
$$d(x_1,x_2)+d(x_2,x_3)+d(x_1,x_3)+d(y_1,y_2)+d(y_2,y_3)+d(y_1,y_3)\leq 4k.$$
As $4k =2\,\diam(C_{2k}\square C_{2k})$, this establishes the lemma.
\end{proof}

We use this maximum possible sum of the pairwise distances between
three vertices of $C_{2k}\square C_{2k}$ together with the radio
condition to determine the minimum distance between every other
label (arranged in increasing order) in a radio labeling of
$C_{2k}\square C_{2k}$.
%*************************Gap Lemma C2k***************************************%
\begin{lemma}\label{Gap lemma even}
Let $c$ be a radio labeling of $C_{2k} \square C_{2k}$. Then for any three vertices
$u,v,w \in V(C_{2k} \square C_{2k})$ satisfying $c(u)< c(v) < c(w)$, we have
$c(w)-c(u) \geq k+2$.
\end{lemma}

\begin{proof}
Since $c(u)$, $c(v)$ and $c(w)$ are radio labels,
$$\begin{aligned}
d(u,v) + |c(v) - c(u)| &\geq 1 + \diam(C_{2k} \square C_{2k}),\\
d(v,w) + |c(w) - c(v)| &\geq 1 + \diam(C_{2k} \square C_{2k}),\text{ and}\\
d(u,w) + |c(w) - c(u)| &\geq 1 + \diam(C_{2k} \square C_{2k}).
\end{aligned}
$$
Summing these inequalities yields
$$\begin{aligned}
d(u,v) + d(v,w) &+ d(u,w) + c(v) - c(u) + c(w) - c(v) + c(w) - c(u) \\
&\geq 3 + 3\,\diam(C_{2k} \square C_{2k}).
\end{aligned}
$$
Furthermore, by Lemma
\ref{distance even}, $d(u,w) + d(v,w) + d(u,w) \leq 2\,\diam(C_{2k}
\square C_{2k})$, so we have
$$2\,\diam(C_{2k} \square C_{2k}) + 2c(w) - 2c(u) \geq 3 + 3\,\diam(C_{2k} \square C_{2k}).$$
As $\diam(C_{2k} \square C_{2k}) = 2k$, it follows that
$$ \begin{aligned}
2(2k) + 2c(w) - 2c(u) &\geq 3 + 3(2k) \\
c(w) - c(u) &\geq \frac{3+2k}{2} = \frac{3}{2} + k.
\end{aligned}
$$
As $c(w) - c(u)$ is an integer, we may conclude that $c(w) - c(u) \geq 2 + k$.
\end{proof}

Knowledge of the size of the minimum gap allowable between the values of every other label makes it possible to calculate the minimum possible span of a radio labeling of $C_{2k}\square C_{2k}$.

%***********************************Lower Bound 2k*****************************%
\begin{thm}For $n=2k$, $rn(C_n \square C_n)\geq \frac{n^2-2}2(k+2)+2$.
 \label{lowerbound2k}
\end{thm}
\begin{proof}
Let $c$ be a radio labeling of $C_{2k} \square C_{2k}$. Rename the
vertices of $C_{2k} \square C_{2k}$ using the set $\{x_1, x_2, \dots
, x_{(2k)^2}\}$ so that $c(x_i) < c(x_j)$ whenever $i < j$. Consider
the lowest possible values of $c(x_i)$ for each $i$. We have $c(x_1)
\geq 1$ and $c(x_2) \geq 2$. From Lemma \ref{Gap lemma even} we know
$c(x_3) \geq c(x_1) +k+2$, and in general,
\[
c(x_i)\geq
       \begin{cases}
          1+\frac {i-1}2(k+2), &\text{ when $i$ is odd}\\
          2+\frac{i-2}2(k+2), &\text{ when $i$ is even.}
        \end{cases}
\]
Thus $rn(C_{2k} \square C_{2k})\geq \text{span}(c)=c(x_{(2k)^2})
\geq 2+\frac{(2k)^2-2}2(k+2)=\frac{n^2-2}2(k+2)+2$.
%Since the vertex set of $C_{2k} \square C_{2k}$ has $4k^2$ elements, we need $\frac{4k^2}{2}=2k^2$ pairs of labels to cover all vertices. By Lemma \ref{Gap lemma even}, we may use at most two values from any interval of length $k+2$ to radio label $C_{2k} \square C_{2k}$. Thus no radio labeling may have a span lower than that resulting from assigning 1 for the smallest label and using $2k^2-1$ consecutive intervals of length $k+2$ for the first $2k^2-1$ pairs of labels.  These $2k^2-1$ intervals then comprise the integers from $1$ to $(2k^2-1)(k+2)$.  The lowest values available for the last pair of labels are thus $(2k^2-1)(k+2)+1$ and $(2k^2-1)(k+2)+2$. Therefore the lowest possible span of any radio labeling is $(2k^2-1)(k+2)+2$. This implies that $rn(C_{2k} \square C_{2k})\geq 2k^3+4k^2-k$.
\end{proof}

A lower bound for the radio number of $C_{2k+1}\square C_{2k+1}$ may be obtained in much
the same way as the lower bound for $C_{2k}\square C_{2k}$.
%***********************distance lemma for C_{2k+1}******************************
\begin{lemma}\label{distance odd} Let $u,v,w \in V(C_{2k+1} \square C_{2k+1})$. Then
$$d(u,v)+d(v,w)+d(u,w)\leq 2\,\diam(C_{2k+1}\square C_{2k+1})+2.
$$
\end{lemma}
\begin{proof} As in the proof of Lemma \ref{distance even}, we write the vertices
of $C_{2k+1} \square C_{2k+1}$ via their components:
$u=(x_1,y_1)$, $v=(x_2,y_2)$, and $w=(x_3,y_3)$.  Here, however, the
sum $d(x_1,x_2)+d(x_2,x_3)+d(x_1,x_3)$ may be as much as $2k+1$
(i.e.,once around the cycle). So
$$\begin{aligned}
d(u,&v)+d(v,w)+d(u,w)\\
&=d\big((x_1,y_1),(x_2,y_2)\big)+d\big((x_2,y_2),(x_3,y_3)\big)+d\big((x_1,y_1),(x_3,y_3)\big)\\
&=d(x_1,x_2)+d(x_2,x_3)+d(x_1,x_3)+d(y_1,y_2)+d(y_2,y_3)+d(y_1,y_3) \\
&\leq 2(2k+1) \\
&=2\,\diam(C_{2k+1}\square C_{2k+1})+2.
\end{aligned}
$$
\end{proof}

%************************************Gap Lemma 2k+1****************************%
\begin{lemma}\label{Gap lemma odd}
Let $c$ be a radio labeling of $C_{2k+1} \square C_{2k+1}$. Then for any three vertices $u,v,w \in V(C_{2k+1} \square C_{2k+1})$ satisfying $c(u)< c(v) < c(w)$, we have $c(w)-c(u) \geq k+1$.
\end{lemma}
The proof of Lemma \ref{Gap lemma odd} is analogous to that of Lemma \ref{Gap lemma even}, with
the substitution of $2\,\diam(C_{2k+1}\square C_{2k+1})+2=4k+2$ for $2\,\diam(C_{2k}\square C_{2k})=4k$.

%******************************************Lower Bound 2k+1***********************%
\begin{thm}\label{lowerbound2k+1}
For $n=2k+1$, $rn(C_n \square C_n)\geq \frac{n^2-1}2(k+1)+1$.
\end{thm}
\begin{proof}
The proof is analogous to that of Theorem \ref{lowerbound2k}, with
the substitution of $k+1$ (from Lemma \ref{Gap lemma odd}) for $k+2$
(from Lemma \ref{Gap lemma even}). Now, for any radio labeling $c$
of $C_{2k+1} \square C_{2k+1}$, we have
\[
c(x_i)\geq
       \begin{cases}
          1+\frac {i-1}2(k+1), &\text{ when $i$ is odd}\\
          2+\frac{i-2}2(k+1), &\text{ when $i$ is even.}
        \end{cases}
\]
As $C_{2k+1} \square C_{2k+1}$ has $(2k+1)^2$ vertices, we conclude
$rn(C_{2k+1} \square C_{2k+1})\geq \text{span}(c)=c(x_{(2k+1)^2})
\geq 1+\frac{(2k+1)^2-1}2(k+1)=\frac{n^2-1}2(k+1)+1$.
\end{proof}

\section{Upper Bound}

Our general approach to establishing the upper bound for $rn(C_n
\square C_n)$ consists of three steps. After some preliminaries,
we define a position function $p:\{0, 1, \dots , n^2-1\} \to V(C_n
\square C_n)$ and argue that $p$ is a bijection. Defining $x_i =
p(i)$ allows us to rename the vertices of $C_n \square C_n$ in what
will be a useful way. Next we give a labeling $c:\{x_0, x_1, \dots ,
x_{n^2-1}\} \to \mathbb Z_+$ for which $c(x_0) < c(x_1) < \cdots <
c(x_{n^2-1})$. We then prove that $c$ is a radio labeling of $C_n
\square C_n$. (The fact that $c(x_i) < c(x_j)$ when $i < j$
simplifies the proof that $c$ is a radio labeling.) It follows that
$rn(C_n \square C_n) \leq \text{span}(c)$.

%We use the next fact to limit the number of vertex pairs for which it
%must be verified that a labeling is a radio labeling.

Recall that any radio labeling $c$ of $G$ must satisfy the radio
condition
$$d(u,v)+|c(u)-c(v)|\geq 1+\text{diam}(G)
$$
for all distinct vertices $u,v \in V(G)$. Once $|c(u)-c(v)| \geq
\text{diam}(G)$, the radio condition is satisfied for $u,v$ and for
any pair of vertices with label difference at least as big as
$|c(u)-c(v)|$. The next remark states this fact precisely, and will
be of use in limiting the number of vertex pairs for which it must
be verified that specific labelings satisfy the radio condition.

\begin{remark}\label{usefact}
Let $c:\{x_0, x_1, \dots , x_{|V(G)|-1}\} \to \mathbb Z_+$ be a labeling
of $G = (V,E)$ satisfying $c(x_0) < c(x_1) < \cdots < c(x_{|V(G)|-1})$.
If $c(x_l)-c(x_k) \geq \diam(G)$ for some $k<l$, then $c$ satisfies the radio condition for all pairs of vertices $x_i, x_j$ with $i\leq k$ and $j \geq l$.
\end{remark}

%Should the difference in two vertices' labels not exceed the diameter of the graph, we will need to calculate the distance between the vertices. We indicate how to calculate distances between vertices of $C_n \square C_n$ below.

In preparation for defining the position function, we employ a common means of representing $C_n$: $V(C_n)=\{v\in\mathbb Z\,|\,0\leq v\leq n-1\}$, with $v,w \in V(C_n)$ adjacent exactly when $v\equiv w \pm 1 \pmod n$. This then gives the expected representation of $V(C_n \square C_n)$ as $\{(v,w)\, |\, 0 \leq v,w \leq n-1\}$. Distance between vertices of $C_n \square C_n$ are calculated as in Remark \ref{distancecalc}.
\begin{remark}\label{distancecalc}
Let $(v_{i_1},w_{i_2}), (v_{j_1},w_{j_2}) \in V(C_n \square C_n)$. Then
$$
\begin{aligned}
d\bigl((v_{i_1},w_{i_2}), &(v_{j_1},w_{j_2})\bigr) \\
&=\min\{|i_1-j_1|,n-|i_1-j_1|\}+\min\{|i_2-j_2|, n-|i_2-j_2|\}.
\end{aligned}
$$
\end{remark}

In Section 2 we establish a lower bound for $rn(C_n \square C_n)$
that depends on the parity of $n$. The upper bound also depends on
the parity of $n$; the next two theorems establish this upper bound.

\vskip 12pt \noindent \textbf{Note}:
All calculations on vertices in pair notation are performed modulo $n$.

\begin{thm} \label{thm:ubeven}
Let $n=2k$. Then $rn(C_n \square C_n) \leq \frac{n^2-2}2(k+2)+2$.
\end{thm}

\begin{proof}
%As in the proof of Theorem \ref{thm:ubodd} we provide a radio labeling with span $2k^3 +4k^2-k$.
Define $p:\{0,1,\dots , n^2-1\} \to \{(v,w)\, |\, 0 \leq v,w \leq n-1\}$ by
\[
p(i)=
       \begin{cases}
         ( r,kr+s), &\text{ when $i \equiv 0 \pmod 4$,}\\
         ( r+k,kr+s+k), &\text{ when $i \equiv 1 \pmod 4$,}\\
         ( r,kr+s+k), &\text{ when $i \equiv 2 \pmod 4$,}\\
         ( r+k,kr+s), &\text{ when $i \equiv 3 \pmod 4$,}\\
        \end{cases}
\]
$$\text{where } r=\left\lfloor \frac i{2n}\right\rfloor \text{ and } s=\left\lfloor \frac i4\right\rfloor \pmod k.
$$
{\em Claim}: $p$ is a bijection. Consider the following
possibilities for the relationship of the indices $i$ and $j$ for $i \neq j$:
\begin{enumerate}
\item $i \not\equiv j \pmod 4$; $i$ and $j$ have opposite parity,
\item $i \not\equiv j \pmod 4$; $i$ and $j$ have the same parity,
\item $i \equiv j \pmod 4$;  $\left\lfloor \frac i{2n}\right\rfloor \neq\left\lfloor \frac
j{2n}\right\rfloor$, and
\item $i \equiv j \pmod 4$; $\left\lfloor \frac i{2n}\right\rfloor=\left\lfloor \frac
j{2n}\right\rfloor$.
\end{enumerate}

In the first case, the first components of $p(i)$ and $p(j)$ agree
exactly when $\left\lfloor \frac i{2n}\right\rfloor=\left\lfloor \frac
j{2n}\right\rfloor+k$. But $\left\lfloor \frac i{2n}\right\rfloor \leq \left\lfloor \frac
{n^2-1}{2n}\right\rfloor=\frac{n-2}2 < k$, so this is impossible.

Suppose $i \not\equiv j \pmod 4$ and $i$ and $j$ have the same
parity. If $\left\lfloor \frac i{2n}\right\rfloor\neq\left\lfloor \frac
j{2n}\right\rfloor$ then the first components of $p(i)$ and $p(j)$ are not equal.
Should $\left\lfloor \frac i{2n}\right\rfloor =\left\lfloor \frac j{2n}\right\rfloor$,
we assume WLOG that $j > i$ and $j$ is congruent to $1$ or $2$ modulo $4$. Examine the second component of $p(j)-p(i)$:
$$
\left( k\left\lfloor\frac j{2n}\right\rfloor +\left\lfloor \frac j4\right\rfloor \pmod k+k\right)
-\left(k\left\lfloor\frac i{2n}\right\rfloor +\left\lfloor \frac i4\right\rfloor \pmod k\right).
$$
This reduces to $k+\left(\left\lfloor \frac j4\right\rfloor-\left\lfloor \frac i4\right\rfloor\right) \pmod k$, which is not $0$. So $p(i) \neq p(j)$.

In the third case, the first components of $p(i)$ and $p(j)$ will be
the same only when they are equivalent $\pmod n$. But the possible
values for $r=\left\lfloor \frac i{2n}\right\rfloor$ never reach $n$, so we may
rule out this eventuality. Finally, in the fourth case, note that
the hypotheses imply that $p(i) = p(j)$ if and only if $s_i=\left\lfloor
\frac i4\right\rfloor \pmod k =s_j=\left\lfloor \frac j4\right\rfloor \pmod k$. But
$|i-j|<2n = 4k$, so $\left|\left\lfloor \frac i4\right\rfloor-\left\lfloor \frac
j4\right\rfloor\right|<k$, thus $s_1 \neq s_2$. Therefore we may conclude that
$p$ is a bijection.

Again, we rename the vertices of $C_n \square C_n$  by agreeing that
$p(i) = x_i$. The labeling is given by $c:\{x_0, x_1, \dots ,
x_{n^2-1}\} \to \mathbb Z_+$ by
\[
c(x_i)=
       \begin{cases}
          1+\frac i2(k+2), &\text{ when $i$ is even,}\\
          2+\frac{i-1}2(k+2), &\text{ when $i$ is odd.}
        \end{cases}
\]
As $c(x_{i+4})-c(x_i)> 2k= \diam(C_n \square
C_n)$ for all $i = 0, 1, \dots , n^2-5$, we again apply Remark
\ref{usefact} to limit the vertex pairs for which we must verify
that $c$ is a radio labeling. This verification consists of two subcases.

\emph{Subcase 1}: Consider first pairs of vertices
$\{x_i,x_j\}$ with $|i-j|\leq 3$ and $\left\lfloor \frac i{2n} \right\rfloor =
\left\lfloor \frac j{2n} \right\rfloor$. For $i_2 = i_1+4m$ and $j_2 =j_2+4m$.
where $m$ is an integer, we have
$d(x_{i_1},x_{j_1})=d(x_{i_2},x_{j_2})$ and
$|c(x_{i_1})-c(x_{j_1})|=|c(x_{i_2})-c(x_{j_2})|$, so consideration
of the pairs for which the distances and the label differences are
shown in the tables below suffices. The first table gives the distances between vertices; the second gives the label differences.
\newline
%Below follows Jason Rennie's trick for putting tables side-by-side.
\begin{minipage}{2.2in}
\begin{center}
\begin{tabular}{l||c|c|c|c|c}
      & $x_1$ & $x_2$ & $x_3$ & $x_4$ & $x_5$   \\ \hline\hline
$x_0$ & $2k$  & $k$   & $k$   &       &         \\ \hline

$x_1$ & &$k$   & $k$   &$2k-1$ &         \\ \hline

$x_2$ &       & & $2k$  & $k-1$ & $k+1$   \\ \hline

$x_3$ &       & &       & $k+1$  & $k-1$
\end{tabular}
\end{center}
  \end{minipage}
  \begin{minipage}{2.5in}
\begin{center}
\begin{tabular}{l||c|c|c|c|c}
      & $x_1$ & $x_2$ & $x_3$ & $x_4$ & $x_5$   \\ \hline\hline
$x_0$ & $2k$  & $k$   & $k$   &       &         \\ \hline

$x_1$ & &$k$   & $k$   &$2k-1$ &         \\ \hline

$x_2$ &       & & $2k$  & $k-1$ & $k+1$   \\ \hline

$x_3$ &       & &       & $k+1$  & $k+1$
\end{tabular}
\end{center}
  \end{minipage}

\noindent Summing the corresponding entries from each table shows that the radio condition is satisfied in all cases.

\emph{Subcase 2}: It remains only to verify that the radio condition
holds for vertices with index differences less than four and indices
near a multiple of $2n$. Specifically, we must calculate
$d(u,v)+|c(u)-c(v)|$ for all vertices $\{u,v\}$ of the form
$\{x_{an-3},\,x_{an}\}$, $\{x_{an-2},\,x_{an}\}$,
$\{x_{an-2},\,x_{an+1}\}$, $\{x_{an-1},\,x_{an}\}$,
$\{x_{an-1},\,x_{an+1}\}$, and $\{x_{an-1},\,x_{an+2}\}$, where $a$
is an even integer. Note that taking $a$ even gives $an=a2k\equiv
0\pmod 4$. Also, as $n\geq 2$, we know that
\[
r=\left\lfloor\frac i{2n}\right\rfloor=
       \begin{cases}
          \left\lfloor\frac{an-t}{2n}\right\rfloor=\frac a2-1, &\text{ for $i=an-t$ and $t=1,2,3$}\\
          \left\lfloor\frac{an+t}{2n}\right\rfloor=\frac a2, &\text{ for $i=an+t$ and $t=0,1,2$.}
        \end{cases}
\]
Also, as $s=\left\lfloor\frac i4\pmod k\right\rfloor$, we know
\[
s=       \begin{cases}
          \frac{ak}2-1\pmod k = k-1, &\text{for $i=an-t$ and $t=1,2,3$}\\
          \frac {ak}2\pmod k=0, &\text{for $i=an+t$ and $t=0,1,2$.}
        \end{cases}
\]
Accordingly, we calculate the position functions and the label values for each vertex of interest, providing the results in the table below.
\begin{center}
\begin{tabular}{c|c|c}
  index $i$ & vertex $x_i$ & label $c(x_i)$ \\\hline\hline
  $an-3$ & $\left(\frac a2-1+k, k\left(\frac a2-1\right)-1\right)$ & $2+\frac{an-4}2(k+2)$ \\\hline
  $an-2$ & $\left(\frac a2-1,k\left(\frac a2-1\right)-1\right)$ & $1+\frac{an-2}2(k+2)$ \\\hline
  $an-1$ & $\left(\frac a2-1+k,\frac{ka}2-1\right)$ & $2+\frac{an-2}2(k+2)$ \\\hline
  $an$ & $\left(\frac a2,\frac{ka}2\right)$ & $1+\frac{an}2(k+2)$ \\\hline
  $an+1$ & $\left(\frac a2+k,\frac{ka}2+k\right)$ & $2+\frac{an}2(k+2)$ \\\hline
  $an+2$ & $\left(\frac a2,\frac{ka}2+k\right)$ & $1+\frac{an+2}2(k+2)$ \\
\end{tabular}
\end{center}

This allows the calculation of the distance between vertices and the
difference of the labels for each vertex pair in question.

\begin{center}
\begin{tabular}{c||c|c|c}
\ vertex pair & distance &  label difference &
distance plus label difference  \\ \hline\hline

$x_{an-3},\,x_{an}$   &   &  $2k+3$  & $>2k+3$ \\
\hline

$x_{an-2},\,x_{an}$   &  $k$  &  $k+2$  &  $2k+2$  \\ \hline

$x_{an-1},\,x_{an}$  &  $k$  &  $k+1$  &  $2k+3$
\\  \hline

$x_{an-2},\,x_{an+1}$  &  $k$  &  $k+3$  & $2k+3$
\\ \hline

$x_{an-1},\,x_{an+1}$  &  $k$  &  $k+2$  &  $2k+2$  \\
\hline

$x_{an-1},\,x_{an+2}$  &   &  $2k+3$  &  $>2k+3$  \\
\end{tabular}
\end{center}
The radio condition is satisfied in all cases.

Finally, we compute the span of this radio labeling:
$$span(c) = c(x_{n^2-1})=2+\frac{n^2-2}2(k+2).
$$
\end{proof}

The proof of Theorem \ref{thm:ubodd} has a similar structure to
that of Theorem \ref{thm:ubeven}, but the position and labeling
functions depend on the parity of $k$.

%************************************************************************************************
%****************************This is where the last upper bound theorem starts.******************
%************************************************************************************************

\begin{thm} \label{thm:ubodd}
Suppose $n=2k+1$. Then $rn(C_n \square C_n) \leq \frac{n^2-1}2(k+1)+1$.
\end{thm}

\begin{proof}
For each of $k$ odd and $k$ even we provide a radio labeling with span $\frac{n^2-1}2(k+1)+1$.
\newline{\em Case 1: $k$ is odd.}
\newline
Define $p:\{0,1,\dots , n^2-1\} \to \{(v,w)\, |\, 0 \leq v,w \leq n-1\}$ by
$$p(i)=\left( ik, r+i\left(\frac{k+1}2 \right)\right), \text{ where }
r=\left\lfloor \frac in \right\rfloor.
$$
We wish to show that $p$ is a bijection. Suppose that $p(i)=p(j)$
for some $i\neq j$. Examining the first components of $p(i)$ and $p(j)$, we
see that $ik \pmod n=jk \pmod n$, i.e.,$ik-jk \equiv 0 \pmod n$. As $k$ and $n$ are relatively prime, $i-j\equiv 0 \pmod n$. The second components of $p(i)$ and $p(j)$
must also be equivalent (mod $n$): this gives
$$
\begin{aligned}
0 &\equiv \left(\left\lfloor \frac in \right\rfloor + i\left(\frac{k+1}2
\right) \right) - \left(\left\lfloor \frac jn \right\rfloor + j\left(\frac{k+1}2 \right) \right) \pmod n \\
&= \left\lfloor \frac in \right\rfloor - \left\lfloor \frac jn \right\rfloor
+(i-j)\left(\frac{k+1}2\right) \pmod n \\
&=\left\lfloor \frac in \right\rfloor - \left\lfloor \frac jn \right\rfloor \pmod n.
\end{aligned}
$$
But $i \neq j$ and $i \equiv j \pmod n$ imply $\left\lfloor \frac in
\right\rfloor - \left\lfloor \frac jn \right\rfloor \not \equiv 0 \pmod n$. Thus
$p(i)\neq p(j)$ for distinct $i,j$ in the domain of $p$, and we may
conclude that $p$ is a bijection.

 We now use the elements of the set
$\{x_0, x_1, \dots , x_{n^2-1}\}$ to rename the vertices of $C_n
\square C_n$  by agreeing that $p(i) = x_i$. Define the labeling
$c:\{x_0, x_1, \dots , x_{n^2-1}\} \to \mathbb Z_+$ by
$$c(x_i) = 1+i\left(\frac{k+1}2\right).
$$
\newline{\em Claim}: The labeling $c$ is a radio labeling of $rn(C_n \square C_n)$.
\newline
To establish our claim we must show that $c$ satisfies the radio condition
$$d(u,v)+|c(u)-c(v)|\geq \text{diam}(G)+1 = 2k+1
$$
for all distinct $u,v \in V(C_n \square C_n)$. Note that $c(x_{i+4})-c(x_i)
\geq \diam(G)=2k$ for all $i=0, \dots n^2-4$, so Remark
\ref{usefact} indicates that we need only verify that $c$ satisfies
the radio condition for vertex pairs $x_i, x_{i+j}$ with $j \leq
3$.

We will examine first pairs of vertices with fixed $r$, i.e., vertices with indices
in $\{an, an+1, \dots , an+n-1\}$ for $a = 0, 1, \dots n-1$.
Subsequently we will show that the radio condition is satisfied for
vertices of the form $x_i$, $x_{i+j}$ where $\left\lfloor \frac in
\right\rfloor \neq \left\lfloor \frac {i+j}n \right\rfloor$ and $j \leq 3$. We will handle the case $n=3$ ($k=1$) separately.

\emph{Subcase 1}: Take $x_i,x_{i+j} \in \{an, an+1, \dots ,
an+n-1\}$ for $a = 0, 1, \dots n-1$. Assume $k>1$.
The distance between $x_i$ and $x_{i+j}$ is given by examining the
position function $p$ and using Remark \ref{distancecalc}.

\begin{center}
\begin{tabular}{c||c|c|c}
vertex pair & $d(x_i,x_{i+j})$ &  $c(x_{i+j})-c(x_i)$ &
$d(x_i,x_{i+j}) + |c(x_j)-c(x_i)|$  \\ \hline\hline

$x_i,\,x_{i+1}$   & $k+\frac{k+1}2$  &  $\frac {k+1}2$  & $2k+1$ \\
\hline

$x_i,\,x_{i+2}$   &  $1+k$  &  $k+1$  &  $2k+2$  \\ \hline

$x_i,\,x_{i+3}$  &  $k-1+\frac {k-1}2$  &  $\frac{3k+3}2$  & $3k$
\\
\end{tabular}
\end{center}
Each sum in the last column is at least $2k+1$ (given $k>1$), so this completes the argument that the radio condition is
satisfied by $c$ for all vertex pairs specified in this subcase.

\emph{Subcase 2}: Consider $x_i$, $x_{i+j}$ with $j \leq 3$ and
$\left\lfloor \frac in \right\rfloor = \left\lfloor \frac {i+j}n \right\rfloor -1$, again assuming $k>1$.
Compare the calculation of $d(x_i,x_{i+j})$ here with the analogous
calculation in Subcase 1. The new condition introduced here, that
$\left\lfloor \frac in \right\rfloor = \left\lfloor \frac jn \right\rfloor -1$, may change
the distance by $\pm 1$ (as $r$ changes by 1 in the second component
of $p(i+j)$). The previous verification that the radio condition
holds thus suffices here for $(x_i,x_{i+2})$ and $(x_i,x_{i+3})$, as
the sum of the distance and the label difference exceeded $2k+1$. We
recalculate $d(x_i,x_{i+1})$:
$$\begin{aligned}
d(x_i,&x_{i+1})=d(x_{an-1},x_{an})\\
&=d\left(\left((an-1)k,(a-1)+(an-1)\frac{k+1}2\right),\left(ank,a+(an)\frac{k+1}2\right)\right)\\
&\equiv d\left(\left(-k,a-\frac{k+3}2\right),(0,a)\right) \\
&= k+\frac{k+3}2.
\end{aligned}
$$
The distance increases; the radio condition is satisfied.

The two subcases show that $c$ is a radio labeling of $C_n \square
C_n$ when $k>1$ ($n>3$). However, $c$ is also a radio labeling of $C_3
\square C_3$. To see this, let $n=3$. Recall $\diam(C_3 \square C_3) = 2$. Note that $p(i)$ and
$p(i+1)$ differ in both components for all $i = 0, 1, \dots , 8$, thus $d(x_i,x_{i+1}) \geq 2$. When $|j-i|\geq 2$, we have $|c(x_j)-c(x_i)| \geq 2$. These two facts ensure
that the radio condition is satisfied by all pairs of vertices of
$C_3 \square C_3$.

This establishes the claim that $c$ is a radio labeling of $C_n \square C_n$ (when $n = 2k+1$ and $k$ is odd). To calculate the span of $c$, we use the fact that $c$ is an increasing function to note that
$$
span(c) = c(x_{n^2-1})=1+(n^2-1)\left(\frac{k+1}2\right) =\left(\frac{n^2-1}2\right)(k+1)+1.
$$
\newline{\em Case 2: $k$ is even.}
\newline
As $C_1 \square C_1$ has only one vertex, we label this vertex 1;
the result follows.

Define $D_i$, the $i$th ``diagonal" of $C_n \square C_n$, to be the set of all vertices $\{(v,w)\,|\,v-w\equiv i \pmod n\}$. We define the position function $p$ onto the vertices of diagonals $D_0, D_1, \dots, D_{n-2}$ first. Define $p(0)= (0,0)$ and $p(1)=(k+1,k)$. Next define
\[
p(i)=\begin{cases}
p(i-2)+\left(\frac k2, \frac k2\right), &i = 2, 3, \dots , 2n-1,\\
p(i-2n)+(k+2,k), &i=2n, 2n+1, \dots , (n-1)n.
\end{cases}
\]
For $i = (n-1)n, (n-1)n+1, \dots , n^2-2$, write $i=(n-1)n+4j+r$, where $r \in \{0,1,2,3\}$. The continuation of the definition of the position function maps these index values to vertices on diagonal $D_{n-1}$:
\[
p(i+1)=
       \begin{cases}
         p(i)+(k-j,k-j), &r = 0, 2,\\
         p(i)+(\frac k2+1+j, \frac k2+1+j), &r=1, 3.
        \end{cases}
\]

\noindent{\em Claim}: $p$ is a bijection.

Consider $\{p(i)\,|\,i=0, 1, \dots , 2n-1\}$. Note $p(0)\in D_0$ and $p(1)\in D_1$. As adding $\left(\frac k2, \frac k2\right)$ to a vertex on $D_i$ yields a vertex on $D_i$, $\{p(i)\,|\,i=0, 1, \dots , 2n-1\}\subset D_0 \cup D_1$. Furthermore, as $\frac k2$ and $n=2k+1$ are relatively prime, $p(i)\neq p(j)$ for $1 \leq i < j \leq 2n-1$. Thus $p|_{i=0, 1, \dots , 2n-1}$ is a bijection onto $D_0\cup D_1$.

Next, observe that $\{p(i)\,|\,i=2n, 2n+1, \dots , 4n-1\}$ shifts $\{p(i)\,|\,i=0, 1, \dots , 2n-1\}$ onto $D_2 \cup D_3$ by adding $(k+2,k)$ to each vertex. Similarly, the assignments $\{p(i)\,|\,i=4n, 4n+1, \dots , 6n-1\}$ are then onto $D_4 \cup D_5$, etc. Thus $p|_{i=0, 1, \dots , (n-1)n-1}$ is a bijection onto the vertices of diagonals $D_0, D_1, \dots , D_{n-2}$.

Finally, consider $p|_{i=(n-1)n, \dots , n^2-1}$. As in the preceding paragraph, $p((n-1)n)$ is a shift of $p((n-3)n)$ onto diagonal $D_{n-1}$. Adding $(\star,\star)$ (where $\star$ is $k-j$ or $\frac k2 + 1 + j$) then ensures that $p(i) \in D_{n-1}$ for $i = (n-1)n+1, \dots , n^2-1$. As $D_{n-1} = \{(0,1), (1,2), \dots , (n-1,0)\}$, the fact that $p|_{i=(n-1)n, \dots , n^2-1}$ is a bijection may be established by considering only the first components of $\{p(i)\,|\,i=(n-1)n, \dots , n^2-1\}$. Yet this is exactly the function $\tau:\{0, 1, \dots , n-1\}\to\{0, 1, \dots , n-1\}$ used by Liu and Zhu to specify an optimal radio labeling of $C_{2k+1}$ for $k$ even \cite{LiuZhu}. Accordingly, the proof that $\tau$ is a permutation also verifies that our adaptation, $p|_{i=(n-1)n, \dots , n^2-1}$, is a bijection onto $D_{n-1}$. Thus $p$ is a bijection onto the vertices of $C_n \square C_n$.

With the claim established, we may rename the vertices of $C_n \square C_n$ by specifying that $p(i) = x_i$. We then define a labeling function $c:V(C_n \square C_n) \rightarrow  \mathbb Z_+$, using one definition for the vertices on diagonals $D_0$, $D_1$, $\dots , D_{n-2}$ and a second for those on $D_{n-1}$. Define $c(x_0) = 1$, and for $i = 0, 1, \dots , (n-1)n-1$, define
\[
c(x_{i+1})=\begin{cases}
         c(x_i)+1, & \text{$i$ even},\\
         c(x_i)+k, & \text{$i$ odd}.
         \end{cases}
\]
For the last $n-1$ vertices, we again write $i=(n-1)n+4j+r$ where $r \in \{0,1,2,3\}$ (for $i=(n-1)n, \dots , n^2-2)$, and use this decomposition to define
\[
c(x_{i+1})=\begin{cases}
         c(x_{i})+2j+1, & r=0, 2,\\
         c(x_{i})+k-2j, & r=1, 3.
         \end{cases}
\]

\noindent{\em Claim}: The labeling $c$ is a radio labeling of $C_n\square C_n$.

 Again we will show that $c$ satisfies the radio condition for pairs of distinct vertices. Applying Remark \ref{usefact} shows that, for vertex pairs on diagonals $D_0$ through $D_{n-2}$, we need only verify that $c$ satisfies the radio condition for vertex pairs $(x_i, x_{i-j})$ with $j \leq 2$ when $i$ is even and $j\leq 3$ when $i$ is odd. On diagonal $D_{n-1}$ it suffices to show $c$ satisfies the radio condition for vertex pairs $(x_i, x_{i-j})$ with $j \leq 3$. We break this verification into subcases.

{\em Subcase} 1: Say $x_i, x_{i-j} \in D_s\cup D_{s+1}$ for some even $s$. \newline
Using Remark \ref{distancecalc} to calculate distances, we see
\[
d(x_i, x_{i-1}) = \begin{cases}
                   2k,& \text{$i$ odd},\\
                   k+1,& \text{$i$ even},
                  \end{cases}
\]
for $i \in \{1, 2, \dots, 2n-1\}$, and $d(x_i, x_{i-2})=k$ for $i \in \{2,3, \dots, 2n-1\}$. For $i$ odd, $i \in \{3, 4, \dots, 2n-1\}$, we have $d(x_i, x_{i-3})=k+1$. We summarize these distance calculations together with the corresponding label difference calculations below.
\begin{center}
\begin{tabular}{c||c|c|c}
  %\hline
  vertex pair &  distance &  label difference & distance + label diff.  \\\hline\hline
  $x_i,x_{i-1}$ ($i$ even)  & $k+1$ & $k$ & $2k+1$ \\\hline
  $x_i,x_{i-1}$  ($i$ odd) & $2k$ & $1$ & $2k+1$ \\\hline
  $x_i,x_{i-2}$  & $k$ & $k+1$ & $2k+1$  \\\hline
  $x_i,x_{i-3}$  ($i$ odd) & $k+1$ & $k+2$ & $2k+3$  \\
  \hline
\end{tabular}
\end{center}

\noindent As the sum in the last column is at least $2k+1$ for each vertex pair, $c$ satisfies the radio condition for these vertex pairs. The extension of the position function in Part 2 of its definition via a constant shift ensures that $c$ also satisfies the radio condition for any vertex pair on diagonals $D_s\cup D_{s+1}$ for $s$ even.

{\em Subcase} 2: Consider $x_i$ and $x_{i+j}$ where $x_i \in D_s\cup D_{s+1}$ and $x_{i+j} \in D_{s+2} \cup D_{s+3}$ for some even $s$. \newline
We examine the sum of vertex distance and label difference for the vertex pairs $(x_{2n},x_{2n-1})$, $(x_{2n},x_{2n-2})$, $(x_{2n+1},x_{2n-1})$, and $(x_{2n+1},x_{2n-2})$. Again taking advantage of the shift employed in the position function's definition, these sums then extend to all vertex pairs under consideration in this case.

To aid in calculating distances, we specify the vertices using the original notation: $x_{2n}=(k+2,k)$; $x_{2n-1}=\left(\frac k2 + 1, \frac k2\right)$; $x_{2n-2}= \left(\frac {3k}2 + 1,\frac {3k}2 + 1\right)$; $x_{2n+1}=(2,2k)$.
The next table gives distances between vertices and label differences. To help the reader follow the distance calculations, we give the contribution of the row indices to the sum first, followed by the contribution of the column indices, and then simplify. (See Remark \ref{distancecalc}.)

\begin{center}
\begin{tabular}{c||c|c}
  %\hline
  vertex pair &  distance &  label difference  \\\hline \hline
  $x_{2n},x_{2n-1}$ & $\left(\frac k2+1\right)+\frac k2 = k+1$ & $k$ \\\hline
  $x_{2n},x_{2n-2}$ & $\left(\frac k2-1\right)+\left(\frac k2+1\right)=k$ & $k+1$  \\\hline
  $x_{2n+1},x_{2n-1}$ & $\left(\frac k2-1\right)+\left(\frac k2+1\right)=k$ & $k+1$  \\\hline
  $x_{2n+1},x_{2n-2}$ & $\left(\frac k2+2\right)+\left(\frac k2-1\right)=k+1$ & $k+2$ \\
  %\hline
\end{tabular}
\end{center}
As the sum of the distance between each pair plus the absolute difference of their label values always exceeds $2k$, the labeling $c$ satisfies the radio condition for these vertex pairs.

{\em Subcase} 3: We consider here the two vertex pairs $(x_{(n-1)n-1},x_{(n-1)n+1})$ and $(x_{(n-1)n-2},x_{(n-1)n+1})$. (These vertex pairs have their first vertex in $D_{n-2}$ and their second in $D_{n-1}$. As the first vertex in $D_{n-1}$, $x_{(n-1)n}$, follows the labeling pattern of the first vertex in all evenly indexed diagonals, we do not need to consider any pair containing it here.) The pair notation for the vertices under consideration is $x_{(n-1)n-2}=(2k-1, 1)$; $x_{(n-1)n-1}=(k-1,k+1)$, and $x_{(n-1)n+1}=(\frac k2 -1, \frac k2)$. This gives the values in the following table.
\begin{center}
\begin{tabular}{c||c|c}
  %\hline
  vertex pair & distance & label difference  \\ \hline\hline
  $x_{(n-1)n-1},x_{(n-1)n+1}$ & $\frac k2+\left(\frac k2+1\right)=k+1$ & $k+1$ \\ \hline
  $x_{(n-1)n-2},x_{(n-1)n+1}$ & $\left(\frac k2+1\right)+\left(\frac k2-1\right)=k$ & $k+2$  \\
  %\hline
\end{tabular}
\end{center}
Both distance plus label difference sums are at least $2k+1$, as required.

{\em Subcase} 4: Say $x_i, x_{i+l} \in D_{n-1}$, with $l \leq 3$. \newline
Recall that to define $c$ for these vertices, we write $i = (n-1)n+4j+r$ for $r\in \{0,1,2,3\}$. The table below giving the results of the necessary calculations for each vertex pair follows from the definitions of $p$ and $c$.

\begin{center}
\begin{tabular}{c||c|c}
  %\hline
  vertex pair(s) & distance & label difference  \\ \hline\hline
  $x_{4j},x_{4j+1}$; $x_{4j+2},x_{4j+3}$ & $2k-2j$ & $1+2j$ \\ \hline
  $x_{4j},x_{4j+2}$; $x_{4j+1},x_{4j+3}$; $x_{4j+2},x_{4(j+1)}$ & $k$ & $1+k$\\\hline
  $x_{4j},x_{4j+3}$ & $k-2j$ & $k+2+2j$ \\ \hline
  $x_{4j+1},x_{4j+2}$; $x_{4j+3},x_{4(j+1)}$ & $k+2+2j$ & $k-2j$  \\ \hline
  $x_{4j+1},x_{4(j+1)}$; $x_{4j+3},x_{4(j+1)+2}$ & $2+2j$ & $2k+1-2j$ \\ \hline
  $x_{4j+3},x_{4(j+1)+1}$ & $k+2$ & $k+3$ \\
\end{tabular}
\end{center}

All sums of the last two columns (distance plus label difference) are at least $2k+1$, so once again we see that $c$ satisfies the radio condition for these vertex pairs. This proves the claim.
\vskip 12pt

Having established that $c$ is a radio labeling on $C_n \square C_n$, it remains only to compute the span of $c$ to obtain an upper bound for $rn(C_n \square C_n)$. We take advantage of the fact that $c(x_{i+2}) = c(x_i)+k+1$ for all even $i$: this gives $c(x_{n^2-1}) = c(x_0)+\frac{n^2-1}2(k+1)$. As $c(x_{n^2-1})=span(c)$, we conclude that $rn(C_n\square C_n) \leq \frac{n^2-1}2(k+1)+1$.

\end{proof}

As each upper bound for the radio number of the Cartesian product of
a cycle with itself is equal to the corresponding lower bound, we
have the radio numbers themselves. That is, Theorems \ref{lowerbound2k}, \ref{lowerbound2k+1}, \ref{thm:ubodd}, and \ref{thm:ubeven} establish our main theorem, Theorem \ref{thmM:main}.
%\begin{thmM} Let $k$ be a nonnegative integer. Then $rn(C_{2k} \square C_{2k})=2k^3+4k^2-k$ and $rn(C_{2k+1} \square C_{2k+1})=2k^3+4k^2+2k+1$.
%\end{thmM}

\section{Additional Comments}

It would be interesting to determine the general relationship
between the radio number of a Cartesian product of graphs and the
radio numbers of the factor graphs (i.e.,between $rn(G\square H)$,
$rn(G)$, and $rn(H)$). For instance, $rn(K_m)=m$ and $rn(K_m \square
K_n) = mn = |V(K_m \square K_n)|$. While this might lead to the hope
that $rn(G\square H) = rn(G)rn(H)$, the result in this paper
together with Liu and Zhu's result on the radio number of cycles
\cite{LiuZhu} shows that this is not the case:
$$\begin{aligned}
rn(C_{2k}) &= k^2+k+2\text{, and} \\
rn(C_{2k+1}) &=
       \begin{cases}
          k^2+k+1, &\text{ when $k$ is even,}\\
          k^2+2k+1, &\text{ when $k$ is odd,}
        \end{cases}
\end{aligned}
$$
whereas
$$\begin{aligned}
rn(C_{2k}\square C_{2k}) &= \bigl(\frac{(2k)^2-2}2\bigr)(k+2)+2\text{, and} \\
rn(C_{2k+1} \square C_{2k+1})&=\bigl(\frac{(2k+1)^2-1}2\bigr)(k+1)+1.
\end{aligned}
$$
From this, we see that $rn(C_n \square C_n)$ is markedly less than
$rn(C_n)^2$.  At this point, not enough is known about radio
numbers of Cartesian products to venture a conjecture regarding the
relationship of $rn(G\square H)$ to $rn(G)$ and $rn(H)$.

\section{Acknowledgements}
This research was initiated under the auspices of an MAA (SUMMA)
Research Experience for Undergraduates program funded by NSA, NSF,
and Moody's, and hosted at CSU Channel Islands during Summer 2006.
The fourth author was partially supported via NSF Grant DMS 0302456,
through the generosity of Dr. Daphne Der-Fen Liu. We are grateful to
all for the opportunities provided. We also thank an anonymous referee for careful reading and helpful suggestions.


\begin{thebibliography}{3}

\bibitem{grids}
Calles, L., Gomez, H., Tomova, M., and Wyels, C., {Bounds for the radio number of square grids}, in preparation.

\bibitem{CEHZ}
Chartrand, G., Erwin, D., Zhang, P. and Harary, F.,
{Radio labelings of graphs}, Bull. Inst. Combin. Appl.,
\textbf{33} (2001), 77--85.

\bibitem{gears}
Fernandez, C., Flores, A., Tomova, M., and Wyels, C., {The radio number of gear graphs}, preprint available at http://front.math.ucdavis.edu/0809.2623.

\bibitem{ladders}
Flores, J., Lewis, K., Tomova, M., and Wyels, C., {The radio number of ladder graphs}, in preparation.

\bibitem{Hale}
Hale, W.K.,{Frequency assignment: theory and application},
Proc. IEEE, \textbf{68} (1980), 1497--1514.

\bibitem{Liu}
Liu, D.D.-F., {Radio number for trees}, Discrete Math. \textbf{308} (2008), no. 7, 1153--1164.

\bibitem{LiuXieC}
Liu, D.D.-F., and Xie, M., {Radio number for square cycles}, Congr. Numer. \textbf{169} (2004), 105–125.

\bibitem{LiuXieP}
Liu, D.D.-F., and Xie, M., {Radio number for square paths}, preprint available at http://www.calstatela.edu/faculty/dliu/ArsComFinal.pdf.

\bibitem{LiuZhu}
Liu, D.D.-F., Zhu, X., {Multilevel distance
labelings for paths and cycles}, SIAM J. Discrete Math. \textbf{19}
(2005), No. 3, 610--621.

\bibitem{prisms}
Ortiz, J.P., Martinez, P., Tomova, M., and Wyels, C., {Radio numbers of some generalized prism graphs}, preprint available at http://faculty.csuci.edu/cynthia.wyels/REU/aids.htm.

\bibitem{SR}
Sooryanarayana, B., and Raghunath, P., {Radio labeling of cube of a cycle}, Far East J. Appl. Math. \textbf{29} (2007), no. 1, 113--147.


\end{thebibliography}
\end{document}